\documentclass[11pt]{amsart}%
\usepackage{amscd,amsmath,latexsym,amsthm,amsfonts,amssymb,graphicx,color,geometry,hyperref}
\usepackage[all]{xy} 
\usepackage[utf8]{inputenc}
\usepackage{amsmath}
\usepackage{amsfonts}
\usepackage{amssymb}
\usepackage{graphicx}%
\providecommand{\U}[1]{\protect\rule{.1in}{.1in}}
\newtheorem{thm}{Theorem}[section]
\newtheorem{cor}[thm]{Corollary}

\newtheorem{lem}[thm]{Lemma}
\newtheorem{pps}[thm]{Proposition}

\theoremstyle{definition}

\newtheorem*{thm*}{Theorem}
\newtheorem*{pps*}{Proposition}
\newtheorem*{lem*}{Lemma}

\theoremstyle{remark}
\numberwithin{equation}{section}

\def\<{\langle}
\def\>{\rangle}
\begin{document}
	
\title[ ]{Complex symmetry of Linear Fractional composition operators on a half-plane }
\author{\textsc{V.V. F\'{a}varo \, P.V. Hai and O.R. Severiano }}
\address[V.V. F\'{a}varo]{Faculdade de Matem\'{a}tica, Universidade Federal de
	Uberlândia, 38.400-902 - Uberl\^{a}ndia.}
\email{vvfavaro@ufu.br}
\address[P. V. Hai]{School of Applied Mathematics and Informatics, Hanoi University of Science and Technology, Vien Toan ung dung va Tin hoc, Dai hoc Bach khoa Hanoi, 1 Dai Co Viet, Hanoi, Vietnam.}
\email{hai.phamviet@hust.edu.vn }

\address [O. R. Severiano]{ Programa Associado de P\'{o}s Gradua\c{c}\~{a}o em
	Matem\'{a}tica Universidade Federal da Para\'{\i}ba/Universidade Federal
de Campina Grande, João Pessoa, Brazil.}
\email{osmar.rrseveriano@gmail.com}
\thanks{V.V. F\'{a}varo is supported by FAPEMIG Grant PPM-00217-18.}
\thanks{O.R. Severiano is postdoctoral fellowship at Programa Associado de Pós Graduação em Matemática UFPB/UFCG, and is
	supported by INCTMat Grant 88887.613486/2021-00. }
\begin{abstract} We investigate the 
bounded composition operators induced by linear fractional self-maps of the right half-plane $\mathbb{C}_+$ on the Hardy space $H^2(\mathbb{C}_+).$ We completely characterize which of these operators are cohyponormal and we find conjugations for the linear fractional composition operators that are complex symmetric. 
\end{abstract}
\maketitle

\section{Introduction}

All along this paper, $\mathbb{C}$ denotes the complex plane, $\mathbb{U}:=\{z\in \mathbb{C}:|z|<1\}$ is the unit disk and $\mathbb{C}_+:=\{z\in \mathbb{C}:\mathrm{Re}(z)>0\}$ is the right half-plane.

Let $\Omega \subset \mathbb{C}$ and let $\mathcal{S}$ be a space of functions defined on $\Omega.$ A composition operator $C_{\phi}$ on $\mathcal{S}$ is an operator acting by composition to the right with a chosen
self-map $\phi$ of $\Omega$, i.e.,
\begin{align*}
C_{\phi}f=f \circ \phi, \quad f \in \mathcal{S}.
\end{align*} 
The self-map $\phi$ is called the \emph{symbol} of the composition operator $C_\phi.$ 
It is well-known that if $\Omega=\mathbb{U}$ and $\phi$ is an analytic self-map of $\mathbb{U}$ then $C_{\phi}$ is bounded on the Hardy space of the unit disk $H^2(\mathbb{U})$ (see \cite{Cowen-MacClauer}).
If $\Omega=\mathbb{C}_+,$ Elliot and Jury established a boundedness criterion for composition operators on the Hardy space of the right half-plane $H^2(\mathbb{C}_+)$ in terms of angular derivative (see \cite[Theorem 3.1]{Elliot-Jury}). As a consequence of this caracterization, we have that the only linear fractional self-maps of $\mathbb{C}_+$
inducing bounded composition operators on $H^2(\mathbb{C}_+)$ are those of the following form
\begin{align}\label{1}
\phi(w)=a w+b, \quad \text{where} \ a>0 \ \text{and} \ \mathrm{Re}(b)\geq0.
\end{align}

Let $\mathcal{L}(\mathcal{H})$ denote the space of all bounded linear operators on a separable complex Hilbert space $\mathcal{H}.$ A \textit{conjugation} on $\mathcal{H}$ is a conjugate-linear
operator satisfying $C^2=I$ and $\left\langle Cx,Cy\right\rangle =\left\langle y,x\right\rangle $ for all $x,y\in \mathcal{H}$. An operator $T\in \mathcal{L}(\mathcal{H})$ is called \emph{cohyponormal} if $\left\| Tx\right\| \leq \left\| T^*x\right\| $ for all $x\in \mathcal{H},$ \textit{normal} if $TT^*=T^*T,$ \textit{self-adjoint} if $T=T^*,$ \textit{unitary} if $TT^*=I=T^*T$ and \textit{complex symmetric} (or \emph{$C$-symmetric}) if there is a conjugation $C$ on $\mathcal{H},$  for which $CT^*C=T.$ In \cite{Osmar}, Noor and Severiano studied the composition operators on $H^2(\mathbb{C}_+)$ induced by linear fractional self-maps of $\mathbb{C}_+.$ They completely characterize the symbols that induce complex symmetric composition operators (see next theorem) and provide a new prove to characterize normal, self-adjoint and unitary composition operators on $H^2(\mathbb{C}_+).$

\begin{thm}\label{2}{\cite[Theorems 2 and 6]{Osmar}} Let $\phi$ be as in \eqref{1}. Then
\begin{enumerate}
\item $C_{\phi}$ is normal on $H^2(\mathbb{C}_+)$ if and only if $a=1$ or $\mathrm{Re}(b)=0.$
\item $C_{\phi}$ is self-adjoint on $H^2(\mathbb{C}_+)$ if and only if $a=1$ and $b\geq 0.$
\item $C_{\phi}$ is unitary on $H^2(\mathbb{C}_+)$ if and only if $a=1$ and $\mathrm{Re}(b)=0.$
\item $C_{\phi}$ is complex symmetric on $H^2(\mathbb{C}_+)$ if and only if $C_{\phi}$ is normal on $H^2(\mathbb{C}_+).$
\end{enumerate}	
\end{thm}	

Noor and Severiano \cite{Osmar}  study the complex symmetry of composition operators on $H^2(\mathbb{C}_+)$. The key of this study is to analyze  the cyclic behavior of the bounded composition operators induced by linear fractional self-maps of $\mathbb{C}_+$, which allowed them to characterize the operators that are complex symmetric. Despite this characterization, they did not exhibit the conjugations for the cases that these operators are complex symmetric.  

In the present paper, we investigate the bounded composition operators on $H^2(\mathbb{C}_+)$ induced by linear fractional self-maps of $\mathbb{C}_+.$ In Theorem \ref{101}, we completely characterize the cohyponormal composition operators induced by linear fractional self-maps of $\mathbb{C}_+.$
As a consequence of this characterization we obtain a new proof of \cite[Theorem 6]{Osmar}. 
For complex symmetry of composition operators on $H^2(\mathbb{C}_+),$ our main results are Theorems \ref{9} and \ref{100}, which provide concrete examples of conjugations for which linear fractional operators are complex symmetric.



\section{preliminaries} 
\subsection{The Hardy space $H^2(\mathbb{C}_+)$} The Hardy space of the right half-plane is the normed space
$$
H^2(\mathbb{C}_+):=\left\lbrace f\in H(\mathbb{C}_+):\left\| f\right\| =  \left(  \sup_{0<x<\infty}\frac{1}{\pi}\int _{-\infty}^{\infty}|f(x+iy)|^2dy\right) ^{1/2}  <\infty\right\rbrace,
$$
and the norm $\left\| \; \cdot \; \right\| $ is induced by an inner product, which makes $H^2(\mathbb{C}_+)$ a Hilbert space. For each $u\in \mathbb{C}_+,$ the function
\begin{align*}
K_{u}(w)=\frac{1}{\overline{u}+w}, \quad w\in \mathbb{C}_+
\end{align*} 
is called \emph{reproducing kernel} for $H^2(\mathbb{C}_+)$ at $u.$ This kernel function has the fundamental property: 
$$\langle f, K_u\rangle=f(u) \textrm{ for each } f\in H^2(\mathbb{C}_+).$$ If $\phi$ is an analytic self-map of $\mathbb{C}_+$ such that $C_{\phi}$ is bounded, then a simple computation gives 
\begin{align}\label{5}
C_{\phi}^*K_{u}=K_{\phi(u)}
\end{align}
for each $u\in \mathbb{C}_+.$

Recall that the classical Hardy space $H^2(\mathbb{U})$ of the unit disk is given by
$$
H^2(\mathbb{U}):=\left\lbrace f\in H(\mathbb{U}):\left\| f\right\|_{\mathbb{U}} =\left( \sup_{0<r<1}\frac{1}{2\pi}\int _{0}^{2\pi}|f(re^{i\theta})|^2d\theta\right)  ^{1/2}<\infty\right\rbrace. 
$$
The space $H^2(\mathbb{U})$ is a Hilbert space when endowed with the norm $\left\| \;\cdot \:\right\|_{\mathbb{U}}.$ 

Let $\gamma:\mathbb{U}\longrightarrow \mathbb{C}_+$ be the mapping defined by $z\mapsto (1-z)/(1+z),$ then $\gamma$ is a self-inverse bijection. The operator $V: H^2(\mathbb{U})\longrightarrow H^2(\mathbb{C}_+)$ defined by
\begin{align}\label{19}
\left( Vf\right) (w)=\frac{\sqrt{2}}{1+w}f\left( \frac{1-w}{1+w}\right) , \quad  f\in H^{2}(\mathbb{U}), w\in \mathbb{C}_+
\end{align}
is an isometric isomorphism whose inverse $V^{-1}:H^2(\mathbb{C}_+)\longrightarrow H^2(\mathbb{U}) $ is given by
\begin{align}
\left( V^{-1}g\right) (z)=\frac{\sqrt{2}}{1+z}g\left( \frac{1-z}{1+z}\right), \quad g \in H^2(\mathbb{C}_+), z\in \mathbb{U}
\end{align}
see \cite[p. 489]{Kumar}. 
\subsection{Linear fractional composition operators}
 We say that a bounded composition operator on $H^2(\mathbb{C}_+)$ is a \emph{linear fractional composition operator} if its inducing symbol has the form
\begin{align}\label{37}
\phi(w)=a w+ b, \ \text{where} \ a>0 \ \text{and} \ \mathrm{Re}(b)\geq 0.
\end{align}
The map $\phi$ is classified according to the value of $a:$
\begin{itemize}
\item $\phi$ is \emph{parabolic} if $a=1$ (and necessarily $b\neq 0$) and it is a
 \emph{parabolic automorphism} if additionally $\mathrm{Re}(b) = 0.$
\item $\phi$ is \emph{hyperbolic} if $a\in (0,1)\cup(1,\infty)$ and it is a  \emph{hyperbolic automorphism} if additionally  $\mathrm{Re}(b) = 0.$
\end{itemize}

In general, there is not a formula for adjoints of compositions operators. However, when the operator is a linear fractional composition operator on $H^2(\mathbb{C}_+)$, then a such formula exists (see \cite[Proposition 1]{Osmar}) as follows:


\begin{pps}\label{3} Let $\phi$ be as in \eqref{37}. Then the adjoint of $C_{\phi}$ on $H^2(\mathbb{C}_+)$ is given by
\begin{align*}
C_{\phi}^{*}=a^{-1}C_{\psi},
\end{align*}
where $\psi(w)=a^{-1}w+a^{-1}\overline{b}.$
\end{pps}

\subsection{Complex symmetry} Let $\mathcal{H}$ and $\mathcal{H}_1$ be separable Hilbert spaces. In \cite[page 1291]{Garc2}, Garcia and Putinar showed that if $T\in \mathcal{L}(\mathcal{H})$ is $C$-symmetric, then there is a natural way to construct complex symmetric operators on $\mathcal{H}_1,$ just consider $T_1\in\mathcal{L}(\mathcal{H}_1)$ with $T_1=U^{-1}TU$ for some isometric isomorphism $U:\mathcal{H}_{1}\longrightarrow\mathcal{H}.$ We present this fact in the following result, which allow us directly conclude whether a given mapping is a conjugation or a complex symmetric operator:

\begin{lem}\label{065} Let $\mathcal{H}$ and $\mathcal{H}_1$ be separable Hilbert spaces. If $T\in\mathcal{L}(\mathcal{H})$ is $C$-symmetric and $T_1\in\mathcal{L}(\mathcal{H}_1)$ with $T_1=U^{-1}TU$ for some isometric isomorphism $U:\mathcal{H}_{1}\longrightarrow\mathcal{H}$, then
\begin{enumerate}
\item[(a)]\label{6} $U_1^{-1}CU_1$ is a conjugation on $\mathcal{H}_{1},$ for any isometric isomorphism $U_1:\mathcal{H}_{1}\longrightarrow\mathcal{H}.$
\item[(b)] \label{7} $T_{1}$ is $U^{-1}CU$-symmetric.	
\end{enumerate}	
\end{lem}


\section{Cohyponormality}
In this section, we completely characterize the linear fractional composition operators that are cohyponormal on $H^2(\mathbb{C}_+).$ Since $C_{\phi}$ is cohyponormal, when $\phi(w)=aw+b$ with $a=1$ or $\mathrm{Re}(b)=0$ (Theorem \ref{2}), then it is enough to consider the case $a\in (0,1)\cup(1, \infty)$ and $\mathrm{Re}(b)>0$. 

\begin{thm}\label{101}
Let $\phi$ be the self-map of $\mathbb{C}_+$ defined by $\phi(w)=aw+b.$ 
\begin{enumerate}
\item[(a)] \label{200} If $a\in (0,1)$ and $\mathrm{Re}(b)> 0,$ then $C_{\phi}$ is not cohyponormal.

\item[(b)] \label{201} If $a\in (1, \infty)$ and $\mathrm{Re}(b)> 0,$ then $C_{\phi}$ is cohyponormal.
\end{enumerate}
\end{thm}
\begin{proof}
(a) First observe that $u=b(1-a)^{-1}$ is a fixed point of $\phi$ in $\mathbb{C}_+.$ By \eqref{5}, it follows that $K_{u}$ is an eigenvector for $C_{\phi}^*.$ Now, suppose that $C_{\phi}$ is cohyponormal. Since $C_{\phi}$ is cohyponormal, $\mathrm{Ker}(C_{\phi}^*-\overline{\lambda} I)\subset \mathrm{Ker}(C_{\phi}-\lambda I)$ for all $\lambda \in \mathbb{C}.$ Thus $K_{u}$ is an eigenvector for $C_{\phi}$ corresponding to the eigenvalue $1.$ Then a simple computation gives  
 $\phi(w)=w$ for all $w\in \mathbb{C}_+,$ which implies 
 $a=1$ and $b=0.$ This contradiction shows that $C_{\phi}$ is not cohyponormal.

(b) Let $f\in H^2(\mathbb{C}_+)$ and $a\in (1,\infty).$ Considering the change of variables $u=ax+\mathrm{Re}(b)$ and $v=ay+\mathrm{Im}(b),$ we obtain
\begin{align}\label{34}
\left\| C_{\phi}f\right\| ^2&=
\sup_{0<x<\infty}\int_{-\infty}^{\infty} |f(ax+\mathrm{Re}(b)+i(ay+\mathrm{Im}(b))|^2dy\\\nonumber
&=\sup_{\mathrm{Re}(b)<u<\infty} \int_{-\infty}^{\infty} \frac{1}{a}|f(u+iv)|^2dv\\\nonumber
&\leq \sup_{a^{-1}\mathrm{Re}(b)<u<\infty} \int_{-\infty}^{\infty} \frac{1}{a}|f(u+iv)|^2dv\\\nonumber
&\leq \sup_{a^{-1}\mathrm{Re}(b)<u<\infty} \int_{-\infty}^{\infty} |f(u+iv)|^2dv.
\end{align}	
On the other hand, by Proposition \ref{3} we have $C_{\phi}^*f=a^{-1}C_{\psi}f.$ Thus the change of variables $r=a^{-1}x+a^{-1}\mathrm{Re}(b)$ and $s=a^{-1}y-a^{-1}\mathrm{Im}(b)$ gives
\begin{align}\label{33}
\left\| C_{\phi}^*f\right\| ^2&=a^{-1}\sup_{0<x<\infty}\int_{-\infty}^{\infty} |f(a^{-1}(x+iy)+a^{-1}\overline{b}|^2dy\\\nonumber
&= a^{-1
}\sup_{0<x<\infty}\int_{-\infty}^{\infty} |f(a^{-1}x+a^{-1}\mathrm{Re}(b)+i(a^{-1}y-a^{-1}\mathrm{Im}(b))|^2dy\\\nonumber
&=  a^{-1
}\sup_{a^{-1}\mathrm{Re}(b)<r<\infty}\int_{-\infty}^{\infty} a|f(r+is)|^2ds\\\nonumber
& =\sup_{a^{-1}\mathrm{Re}(b)<r<\infty}\int_{-\infty}^{\infty} |f(r+is)|^2ds.
\end{align}	
From \eqref{34} and \eqref{33}, we have $\left\| C_{\phi}f\right\| \leq \| C_{\phi}^*f\|,$ for all $f\in H^2(\mathbb{C}_+).$
\end{proof}

Combining Theorem \ref{2} and Theorem \ref{101}, we obtain

\begin{cor} Let $\phi$ be the self-map of $\mathbb{C}_+$ defined by $\phi(w)=aw+b.$ 
Then $C_{\phi}$ is cohyponormal on $H^2(\mathbb{C}_+)$ if and only if $a\geq 1$ or $\mathrm{Re}(b)=0.$ 
\end{cor}


\begin{cor}
For $a>0$ and $\mathrm{Re}(b)\geq 0,$ let $\phi$ be the self-map of $\mathbb{C}_+$ defined by $\phi(w)=aw+b.$ Then $C_{\phi}$ is complex symmetric on $H^2(\mathbb{C}_+)$ if and only if $a=1$ or $\mathrm{Re}(b)=0.$
\end{cor} 
\begin{proof} Suppose that $C_{\phi}$ is complex symmetric. If $a\in (0,1),$ then by Theorem \ref{101} $aC_{\phi}^*=C_{\psi}$ is cohyponormal, where $\psi(w)=a^{-1}w+a^{-1}\overline{b}.$ Since $C_{\psi}$ is simultaneously cohyponormal and complex symmetric, then $C_{\psi}$ is normal. By Theorem \ref{2}, we obtain $a^{-1}=1$ or $\mathrm{Re}(a^{-1}\overline{b})=0.$ Since $a\in (0,1),$ we have $\mathrm{Re}(b)=0.$ Similarly, if $a\geq 1$ then $C_{\phi}$ is simultaneously cohyponormal and complex symmetric, and hence $C_{\phi}$ is normal. Therefore, $a=1$ or $\mathrm{Re}(b)=0.$ Conversely, since normal operators are complex symmetric, it follows from Theorem \ref{2} that $C_{\phi}$ is complex symmetric whenever $a=1$ or $\mathrm{Re}(b)=0.$ 
\end{proof}

\section{Complex Symmetry} 

In this section, we deal with the problem of exhibit conjugations for the linear fractional composition operators that are complex symmetric.

\subsection{Parabolic self-maps of  $\mathbb{C}_+$} 

Here we explicit a concrete conjugation $C$ for which the compositions operators induced by parabolic symbols are $C$-symmetric.




Let $J: H^2(\mathbb{C}_+)\longrightarrow H^2(\mathbb{C}_+)$ be the mapping defined by 
\begin{align}\label{103}
\left( Jf\right) (w)=\overline{f\left( \overline{w}\right) }, \quad  f \in H^2(\mathbb{C}_+), w\in \mathbb{C}_+.
\end{align}
It is easy to check that $J$ is a conjugation.

Next result shows that the composition operators induced by parabolic symbols are $J$-symmetric.

\begin{thm}\label{9}
Let $\phi$ be an analytic self-map of $\mathbb{C}_+$ such that $C_{\phi}$ is bounded on $H^2(\mathbb{C}_+).$ Then $C_{\phi}$ is $J$-symmetric if and only if $\phi$ is a parabolic symbol.
\end{thm}
\begin{proof} Suppose that $C_{\phi}$ is $J$-symmetric. Then  $JC_{\phi}^*K_{u}=C_{\phi}JK_{u}
$ for each $u\in \mathbb{C}_+.$ Since $JK_{u}=K_{\overline{u}},$ it follows that
\begin{align}\label{15}
K_{\overline{u}}(\phi(w))=K_{\overline{\phi(u)}}(w), \quad  w\in \mathbb{C}_+.
\end{align}
If $u=1$ in \eqref{15}, then
\begin{align*}
\frac{1}{ 1+\phi(w) }=\frac{1}{ \phi(1) + w },
\end{align*}
and hence $\phi(w)=w+\overline{\phi(1)}-1,$ for all $w\in \mathbb{C}_+.$ 

Conversely, suppose that $\phi$ is a parabolic symbol, that is $\phi(w)=w+b$ with $\mathrm{Re}(b)\geq 0.$ Then $(C_{\phi}^*f)(w)=f(w+\overline{b})$, for each $f\in H^2(\mathbb{C}_+),w\in \mathbb{C}_+.$ Hence
\begin{align*}
\left( JC_{\phi}Jf\right) (w)=\overline{C_{\phi}Jf(\overline{w})}=\overline{Jf(\phi(\overline{w}))}=f(w+\overline{b})=(C_{\phi}^*f)(w),
\end{align*}	
which shows that $C_{\phi}$ is $J$-symmetric.	
\end{proof}

\subsection{Hyperbolic automorphism of $\mathbb{C}_+$} 
\label{8}

By Theorem \ref{2}, among the hyperbolic symbols only the automorphisms induce complex symmetric composition operators. Here, we exhibit a concrete conjugation $C$ for which these operators are $C$-symmetric.

Narayan \emph{et al.} \cite{Narayan} showed that if $c\in (-1,1),$ $$\Psi_c(z):=\frac{\sqrt{1-|c|^2}}{1-cz}, \Phi_c(z):=\frac{c-z}{1-cz}$$ and $J_{\mathbb{U}}$ is the conjugation defined on $H^2(\mathbb{U})$ by $(J_{\mathbb{U}}f)(z)=\overline{f(\overline{z})}$ then $J_{\mathbb{U}}T_{\Psi_c}C_{\Phi_c}$ is a conjugation on $H^2(\mathbb{U}).$ By Lemma \ref{065}, the mapping 
\begin{align}\label{10}
W_c:=VJ_{\mathbb{U}}T_{\Psi_c}C_{\Phi_c}V^{-1}
\end{align} 
is a conjugation on $H^2(\mathbb{C}_+).$ In the following result we determine an expression for $W_c$ when $c=0.$
\begin{pps}
If $c=0$ in \eqref{10} then 
\begin{align}
(W_0f)(w)=\frac{1}{w}\overline{f\left( \frac{1}{\overline{w}}\right) },\quad f\in H^2(\mathbb{C}_+), w\in \mathbb{C}_+.
\end{align}
\end{pps}
\begin{proof}
If $c=0$ then $\Psi_0\equiv 1$ and $\Phi_{0}(z)=-z$ for each $z\in \mathbb{U}.$ Hence
$T_{\Psi_0}C_{\Phi_0}=C_{\Phi_0}.$ Now observe that for each $f\in H^2(\mathbb{C}_+)$ and $w\in \mathbb{C}_+,$ we have
\begin{align*}
(W_0f)(w)&=(VJ_{\mathbb{U}}C_{\Phi_0}V^{-1}f)(w)= \frac{\sqrt{2}}{1+w}J_{\mathbb{U}}C_{\Phi_0}V^{-1}f\left( \gamma^{-1}(w)\right) \\
&= \frac{\sqrt{2}}{1+w}\overline{C_{\Phi_0}V^{-1}f\left(\overline{ \gamma^{-1}(w)}
\right)}=  \frac{\sqrt{2}}{1+w}\overline{V^{-1}f\left(-\overline{ \gamma^{-1}(w)}
	\right)}\\
&=\frac{\sqrt{2}}{1+w}\frac{\sqrt{2}}{1- \gamma^{-1}(w)}\overline{f\left(\gamma(-\overline{ \gamma^{-1}(w)})
	\right)}.
\end{align*}
Since 
\begin{align*}
1-\gamma^{-1}(w)=1-\frac{1-w}{1+w}=\frac{1+w-1+w}{1+w}=\frac{2w}{1+w}
\end{align*}
and 
\begin{align*}
\tau(-\overline{ \gamma^{-1}(w)})=\frac{1+\overline{ \gamma^{-1}(w)}}{1-\overline{ \gamma^{-1}(w)}}=\frac{1+\frac{1-\overline{w}}{1+\overline{w}}}{1-\frac{1-\overline{w}}{1+\overline{w}}}=\frac{1+\overline{w}+1-\overline{w}}{1+\overline{w}-1+\overline{w}}=\frac{1}{\overline{w}}
\end{align*}
we obtain 
\begin{align*}
(W_0f)(w)=\frac{\sqrt{2}}{1+w}\frac{\sqrt{2}(1+w)}{2w}\overline{f\left( \frac{1}{\overline{w}}\right) }=\frac{1}{w}\overline{f\left( \frac{1}{\overline{w}}\right) }.
\end{align*}
This completes the proof.
\end{proof}

In the next proposition, we completely describe the $W_0$-symmetric operators $C_{\phi}.$

\begin{thm}\label{26} Let $\phi$ be an analytic self-map of $\mathbb{C}_+$ such that $C_{\phi}$ is bounded on $H^2(\mathbb{C}_+).$ Then $C_{\phi}$ is $W_0$-symmetric if and only if $\phi(w)=a w$ with $a>0.$
\end{thm}
\begin{proof} Suppose that $C_{\phi}$ is $W_0$-symmetric. Thus $W_0C_{\phi}^*K_{u}=C_{\phi}W_0K_{u}$ for each $u\in \mathbb{C}_+.$ Since $(W_0K_u)(w)=\left( \frac{1}{w}\right) \overline{K_u\left( \frac{1}{\overline{w}}\right) },$ we have
\begin{align}\label{20}
\left( \frac{1}{w}\right)\overline{K_{\phi(u)}\left( \frac{1}{\overline{w}}\right) }=\left( \frac{1}{\phi(w)}\right)\overline{K_u\left( \frac{1}{\overline{\phi(w)}}\right) }, \quad w\in \mathbb{C}_+.
\end{align}
If $u=1$ in \eqref{20}, then
\begin{align*}
\frac{1}{\phi(1)w+1}=\frac{1}{\phi(w)+1}
\end{align*}
and hence $\phi(w)=\phi(1)w.$ Since $C_{\phi}$ is bounded, $\phi(1)>0$ (see \eqref{1}). Conversely, suppose that $\phi(w)=a w$ for some $a>0.$ If $\psi(w)=a^{-1}w$ then $C_{\phi}^*=a^{-1}C_{\psi}$ (see Proposition \ref{3}). Hence, for each $f\in H^2(\mathbb{C}_+)$ and $w\in \mathbb{C}_+$ we obtain
\begin{align*}
\left( W_0C_{\phi}W_0f\right)(w)&= \left( \frac{1}{w}\right)\overline{\left( C_{\phi}W_0f\right) \left( \frac{1}{\overline{w}}\right) }\\
&=  \left( \frac{1}{w}\right) \overline{W_0f\left( \frac{a}{\overline{w}}\right) }\\
&=  \left( \frac{1}{w}\right)  \overline{\left( \frac{\overline{w}}{a}\right) \overline{f\left( \frac{w}{a}\right)} }\\
&= \left( C_{\phi}^*f\right) (w).
\end{align*}
Hence $W_0C_{\phi}W_0=C_{\phi}^*.$
\end{proof}	

Let $\phi$ be a hyperbolic automorphism of $\mathbb{C}_+,$ that is, $\phi(w)=aw+b$ with $a\in (0,1)\cup (1, \infty)$ and $\mathrm{Re}(b)=0.$ A simple computation shows that $\phi$ can be written as
\begin{align}\label{21}
\phi=\tau \circ \varphi\circ \tau^{-1},
\end{align}
where $\varphi(w)=a w$ and $\tau(w)=w+(a-1)^{-1}b.$ In particular, \eqref{21} implies 
$C_{\phi}=C_{\tau}^{-1}C_{\varphi}C_{\tau}.$ 
\begin{cor}\label{11}
Let $\phi(w)=a w+b$ with $a\in (0,1)\cup(1,\infty)$ and $\mathrm{Re}(b)=0.$ If $\tau(w)=w+(a-1)^{-1}b$ and $W_{a, b}:=C_{\tau}^{-1}W_0C_{\tau},$ then 
\begin{enumerate}
\item[(a)] \label{28} $W_{a, b}$ is a conjugation.
\item[(b)] \label{27}$C_{\phi}$ is $W_{a, b}$-symmetric on $H^2(\mathbb{C}_+).$
\end{enumerate} 
\end{cor}
\begin{proof} (a) Since $C_{\tau}$ is unitary and $W_0$ is a conjugation, it follows from Lemma \ref{065} that $W_{a, b}=C_{\tau}^{-1}W_0C_{\tau}$ is a conjugation. 

(b) By Theorem \ref{26}, if $\varphi(w)=aw$ then $C_{\varphi}$ is $W_0$-symmetric. Hence 
\begin{align*}
W_{a, b}C_{\phi}W_{a, b}&=C_{\tau}^{-1}W_0C_{\tau}(C_{\tau}^{-1}C_{\varphi}C_{\tau})C_{\tau}^{-1}W_0C_{\tau}\\
&= C_{\tau}^{-1}(W_0C_{\varphi}W_0)C_{\tau}\\
&= C_{\tau}^{-1}C_{\varphi}^*C_{\tau}\\
&= C_{\tau}^{-1}(a^{-1}C_{\varphi^{-1}})C_{\tau}\\
&= a^{-1}C_{\psi},
\end{align*}
where $\psi=(\tau\circ \varphi\circ \tau^{-1})^{-1}.$ A straightforward computation gives $\psi(w)=a ^{-1}w+a^{-1}\overline{b}.$ Since, by Proposition \ref{3},
$C_{\phi}^*=a ^{-1}C_{\psi}$, we get $W_{a, b}C_{\phi}W_{a, b}=C_{\phi}^*.$
\end{proof}	
Since $\tau(w)=w+ir$ with $r\in \mathbb{R}$ induces a unitary operator, it follows from Lemma \ref{065} that $J_r:=C_{\tau}^{-1}W_0C_{\tau}$ is a conjugation on $H^2(\mathbb{C}_+).$ Using this notation, $W_{a,b}$ becomes $J_{(a-1)^{-1}\mathrm{Im}(b)}.$ Now we completely characterize the composition operators on $H^2(\mathbb{C}_+)$ that are $J_r$-symmetric.

\begin{thm}\label{100}Let $\phi$ be an analytic self-map of $\mathbb{C}_+$ such that $C_{\phi}$ is bounded on $H^2(\mathbb{C}_+).$ Then $C_{\phi}$ is $J_{r}$-symmetric for some $r\in \mathbb{R}$ if and only if $\phi$ is a hyperbolic automorphism.
\end{thm}
\begin{proof} Supoose that $C_{\phi}$ is $J_r$-symmetric. Then $W_0C_{\tau}C_{\phi}C_{\tau}^{-1}W_0=(C_{\tau}C_{\phi}C_{\tau}^{-1})^{*},$ which shows that the composition operator $C_{\tau^{-1}\circ \phi \circ \tau}=C_{\tau}C_{\phi}C_{\tau}^{-1}$ is $W_0$-symmetric. By Theorem \ref{26}, there is $a>0$ such that $(\tau^{-1}\circ \phi \circ \tau)(w)=a w$ for each $w\in \mathbb{C}_+.$ Now a simple computation shows that $\phi(w)=aw +i(1-a)r.$ Conversely, in view of the Corollary \ref{11}, each hyperbolic automorphism is $J_{(a-1)^{-1}\mathrm{Im}(b)}$-symmetric for convenient $a$ and $b.$ 
\end{proof}	

\section{Adjoint}
Inspired by Proposition \ref{3} we provide a new formula for adjoints of composition o\-pera\-tors induced by linear fractional self-maps of $\mathbb{C}^+.$ 

\begin{thm}\label{30} Let $\phi(w)=aw+b$ with $a>0$ and $\mathrm{Re}(b)\geq 0.$ If $\varphi(w)=a^{-1}w$ and $J$ is the conjugation defined in \eqref{103}, then $U_a:=a^{-1/2}C_{\varphi}J$ is an isometric conjugate-linear operator  on  $H^2(\mathbb{C}^+)$ and 
\begin{align}\label{29}
	U_aC_{\phi}U_a=C_{\phi}^*.
	\end{align}
\end{thm}

\begin{proof} It is easy to see that $U_a$ is a conjugate-linear operator. By Proposition \ref{3}, we have 
\begin{align*}
(a^{-1/2}C_{\varphi})^*=a^{-1/2}aC_{\varphi^{-1}}=a^{1/2}C_{\varphi^{-1}},
\end{align*}
moreover $$a^{-1/2}C_{\varphi}a^{1/2}C_{\varphi^{-1}}=I=a^{1/2}C_{\varphi^{-1}}a^{-1/2}C_{\varphi}.$$ These equalities imply that $(a^{-1/2}C_{\varphi})^*=(a^{-1/2}C_{\varphi})^{-1},$ and hence $a^{-1/2}C_{\varphi}$ is unitary. Now observe that if $f,g\in H^2(\mathbb{C}^+)$ then
\begin{align*}
\left\langle U_af,U_ag\right\rangle =\langle (a^{-1/2}C_{\varphi})Jf,(a^{-1/2}C_{\varphi})Jg\rangle =
\left\langle Jf,Jg\right\rangle =\left\langle g,f\right\rangle 
\end{align*}
which implies that $U_a$ is an isometry. To show \eqref{29}, first note that for each $f\in H^2(\mathbb{C}^+),$ we have
\begin{align*}
U_aC_{\phi}U_af(w)&=a^{-1/2}C_{\varphi}JC_{\phi}a^{-1/2}C_{\varphi}Jf(w)=a^{-1}JC_{\phi}C_{\varphi}Jf(a^{-1}w)\nonumber\\
&= \overline{a^{-1}C_{\phi}C_{\varphi}Jf(a^{-1}\overline{w})}=\overline{a^{-1}C_{\varphi}Jf(a(a^{-1}\overline{w})+b)}\nonumber\\
&=\overline{a^{-1}Jf(a^{-1}\overline{w}+a^{-1}b)}=a^{-1}f(a^{-1}w+a^{-1}\overline{b})\nonumber.
\end{align*}
By Proposition \ref{3},  we conclude that $U_aC_{\phi}U_a=C_{\phi}^*$.
\end{proof}



\end{document}